\documentclass[a4paper]{amsart}
\usepackage{amsmath, amssymb,amsthm}

\usepackage{graphicx}
\usepackage{txfonts}
\usepackage{bbm}
\usepackage{subfigure}
\usepackage{enumerate}
\usepackage{color}
\usepackage{caption}
\usepackage{accents}

\newtheorem{theorem}{Theorem}

\newtheorem{definition}[theorem]{Definition}
\newtheorem{lemma}[theorem]{Lemma}

\newtheorem*{openproblem}{Open problem}

\theoremstyle{remark}
\newtheorem{remark}[theorem]{Remark}
\newtheorem{example}[theorem]{Example}

  \newcommand{\F}{\mathcal{F}}

 \newcommand{\W}{\mathcal{W}}

 \renewcommand{\phi}{\varphi}

\newcommand{\E}{\mathbb{E}}
\renewcommand{\P}{\mathbb{P}}
\newcommand{\N}{\mathbb{N}}

\newcommand{\R}{\mathbb{R}}

\renewcommand{\c}{\mathfrak c}

\newcommand{\be}{\begin{equation}}
\newcommand{\ee}{\end{equation}}

\newcommand{\bea}{\begin{eqnarray}}
\newcommand{\bes}{\begin{subequations}}
\newcommand{\ees}{\end{subequations}}
\newcommand{\bgt}{\begin{gather}}
\newcommand{\egt}{\begin{gather}}
\newcommand{\eea}{\end{eqnarray}}
\newcommand{\beaa}{\begin{eqnarray*}}
\newcommand{\eeaa}{\end{eqnarray*}}

\renewcommand{\W}{{\mathbb W}}

\usepackage{bbm}

\newcommand{\1}{\mathbbm 1}

\renewcommand{\epsilon}{\varepsilon}

\DeclareMathOperator{\id}{Id}
\DeclareMathOperator{\proj}{proj}

\DeclareMathOperator{\Var}{Var}

\newcommand{\fourIdx}[5]{%
\setbox1=\hbox{\ensuremath{^{#1}}}%
 \setbox2=\hbox{\ensuremath{_{#2}}}%
 \setbox5=\hbox{\ensuremath{#5}}%
 \hspace{\ifnum\wd1>\wd2\wd1\else\wd2\fi}%
 \ensuremath{\copy5^{\hspace{-\wd1}\hspace{-\wd5}#1\hspace{\wd5}#3}%
 _{\hspace{-\wd2}\hspace{-\wd5}#2\hspace{\wd5}#4}%
 }}

\numberwithin{equation}{section}
\numberwithin{theorem}{section}

\renewcommand{\supset}{\supseteq}

\begin{document}

\title{Tranport estimates for random measures in dimension one}

\author{ Martin  Huesmann} \thanks{The author thanks Theo Sturm for valuable discussions and gratefully acknowledges funding by the CRC
  1060 and support by the Hausdorff center for mathematics.}  \date{\today}
  \begin{abstract}

 We show that there is a sharp threshold in dimension one for the transport cost between the Lebesgue measure $\lambda$ and an invariant random measure $\mu$ of unit intensity to be finite. We show that for \emph{any} such random measure the $L^1$ cost are infinite provided that the first central moments $\E[|n-\mu([0,n))|]$ diverge. Furthermore, we establish  simple and sharp criteria, based on the variance of $\mu([0,n)]$, for the $L^p$ cost to be finite for $0<p<1$.
 
\medskip

\end{abstract}
\maketitle
\section{Introduction}

In \cite{HuSt13, Hu15} it was shown that there is a unique optimal coupling between the Lebesgue measure $\lambda^d$ on $\R^d$ and an invariant random measure $\mu$ on $\R^d$ of unit intensity \emph{provided} that the \emph{asymptotic mean transportation cost}
\begin{align}\label{eq:cost1}
\mathfrak c_\infty = \liminf_{n\to\infty}\inf_{q\in \mathsf{Cpl}(\lambda^d,\mu)} \frac1{n^d} \E\left[\int_{\R^d\times [0,n)^d} \vartheta(|x-y|) \ q_\omega(dx,dy)\right] 
\end{align}
is finite, where $\mathsf{Cpl}(\lambda^d,\mu)$ denotes the set of all couplings between $\lambda^d$ and $\mu$ and $\vartheta:\R_+\to\R_+$ is a strictly increasing and diverging function. Moreover, as the optimal coupling $\hat q$ is concentrated on the graph of a random map $T$, i.e.\ $\hat q=(id,T)_*\lambda^d$, a posteriori it can be shown that 
\begin{align}\label{eq:cost2}\mathfrak c_\infty = \inf_{S, S_*\lambda^d=\mu}\E[\vartheta(|0-S(0)|)].\end{align}
In principle, these results give a blackbox construction of allocations and invariant couplings suitable for applications, e.g.\ modelling of cellular structure via Laguerre tessellation \cite{LaZu08} (and references therein) or the recent construction of unbiased shifts \cite{LaMoTh14}.
However, both conditions \eqref{eq:cost1} and \eqref{eq:cost2} are difficult to verify, mainly, because optimal couplings are highly non-local objects. For instance, consider the optimal semicoupling (cf.\ Section \ref{s:prel}) between $\lambda^d$ and a Poisson point process on $B_n=[0,n)^d$. It is an open problem to estimate the amount of mass that is transported from outsided of $B_n$ into $B_n$, for fixed $n$ as well as aysmptotically as $n$ tends to $\infty.$

The aim of this note is to give in dimension one sharp and easily checkable conditions for the asymptotic mean transportation cost to be finite. 
For ease of exposition, in this note we focus on $L^p$ cost, i.e.\ we consider $\vartheta_p(r):=r^p$ for $p>0$, and put
\begin{align*}
\mathfrak{c}_\infty(p) = \inf_{S, S_*\lambda^d=\mu}\E[\vartheta_p(|0-S(0)|)]=\inf_{S, S_*\lambda^d=\mu}\E[|S(0)|^p]. 
\end{align*}

We denote by $\Var(Z)$ the variance of a random variable $Z$. We say that a random  measure $\mu$ satisfies a CLT if the sequence $\left((\mu([0,n))-\E[\mu([0,n))])/\sqrt{\Var(\mu([0,n)))}\right)_n$ weakly converges to a standard normal distribution. We say a random measure $\mu$ has a regular variance if $f(n):=\Var(\mu([0,n)))$ satisfies 
$$\lim_{n\to\infty}\frac{a_n}{n}=0\quad \Rightarrow \quad \lim_{n\to\infty} \frac{f(a_n)}{f(n)}=0.$$
Our first result states

\begin{theorem}\label{thm:p<1} Fix $0<p<1$ and let $\mu$ be an invariant random measure of unit intensity.
 \begin{itemize}
  \item[i)] If $\limsup_{n\to\infty} \sqrt{\Var(\mu([0,n)))} \cdot n^{p-1}=0$, then $\c_\infty(q)<\infty$ for all $0<q<p.$
\item[ii)] Assume that $\mu$ has a regular variance and satisfies a CLT. If $\limsup_{n\to\infty} \sqrt{\Var(\mu([0,n)))} \cdot n^{p-1}>0$, then $\c_\infty(p)=\infty.$
 \end{itemize}
\end{theorem}

For the question of finiteness of $\c_\infty(p)$ or otherwise only the tail of $\vartheta_p$ is relevant. Therefore, $\c_\infty(p)=\infty$ implies $\c_\infty(p')=\infty$ for all $p'>p$ (see also \cite[Lemma 5.1]{HuSt13}).

\begin{remark}
 $\mu$ has a  regular variance, if for example  $f$ is convex (recall $f(n)=\Var(\mu([0,n)))$) or if $f$ is concave and there is $p>0$ such that $f^p$ is convex. Indeed, assume that $f$ is convex and assume for contradiction that $1\geq\liminf_{n\to\infty}  \frac{f(a_n)}{f(n)}\geq c >0$. Then, we have (denoting by $g$ the concave inverse function of $f$, i.e. $g\circ f=f\circ g=\id$, with $f(0)=g(0)=0$) for large $n$ and some $c'<c\leq 1$
\begin{align*}
 a_n\geq g(c' f(n)) = g(c' f(n)+ (1-c')0)\geq c'g(f(n))=c'n,
\end{align*}
which is a contradiction to $a_n\in o(n).$ In the second case we can use the same argument by considering $f=(f^p)^{1/p}$ and using the monotonicity of $x\mapsto x^{1/p}.$
\end{remark}

Formally taking $p=1$ in Theorem \ref{thm:p<1} ii) indicates that $c_\infty(1)$ might be infinite if $\limsup_{n\to\infty}\sqrt{\Var(\mu([0,n)))}=\infty.$ Unfortunately, the proof of Theorem \ref{thm:p<1} breaks down at $p=1$. However, following Liggett \cite[Section 3]{Li02} and combining this with \cite[Proposition 4.5]{LaTh09} we get

\begin{theorem}\label{thm:p=1}
Let $\mu$ be an invariant random measure of unit intensity. If $\limsup_{n\to\infty} \E[|n-\mu([0,n))|]=\infty$, then $\c_\infty(1)=\infty$. 
\end{theorem}
 Note that if $\mu$ satisfies a CLT in $L^1$  the expression in the last Theorem behaves like $\limsup_{n\to\infty}\sqrt{\Var(\mu([0,n)))}.$

Here are a few examples to which our results apply:

\begin{itemize}
 \item[i)] The Poisson point process has finite transport cost iff $p<1/2$. In particular, we recover the second part of Theorem 3.1 of \cite{Li02}.
\item[ii)] Invariant determinantal random point fields \cite{So00} yield a wide and well studied class of random measures to which our results apply. Many of them satisfy a central limit theorem \cite{So02}. The behaviour of $\Var(\mu([0,n)))$ can be expressed nicely via the integral kernel \cite[Lemma 6]{So00}. For instance the determinantal random point field associated to the sine kernel 
$$K(x,y)=\frac{\sin(\pi(x-y))}{\pi(x-y)}$$
satisfies $\Var(\mu([0,n)))\sim \log(n)$. Hence, the transport cost are finite iff $p<1$ (see next point for the only if statement). This behaviour of the variance is not prototypical for determinantal point processes; for each $0<\beta<1$ there is a determinantal point process with $\Var(\mu([0,n)))\sim n^\beta$, see the last paragraph of Section 3 in \cite{So00}.
\item[iii)] The $\mathsf{Sine}_\beta$ point processes introduced in \cite{VaVi09} appear as the limit of the bulk of eigenvalues of $\beta$-ensembles. $\mathsf{Sine}_\beta$ are translation invariant, satisfy a central limit theorem \cite{KrVaVi12} and $\Var(\mu([0,n)))\sim 1/\beta\log(n)$. From the large deviation result \cite{HoVa15} it is possible to deduce that the assumption of Theorem \ref{thm:p=1} is satisfied. Hence, the transport cost are finite iff $p<1.$ Note that $\mathsf{Sine}_2$ is the determinantal process associated to the sine kernel.
\end{itemize}

A natural interpretation of the results is to think of $p*:=\sup\{p, \mathfrak c_\infty(p)<\infty\}$ as a measure of regularity of the random measure. For example in the case of the sine kernel process the repulsion of the particles causes a rigid behaviour reflected in the logarithmic growth of the variance, and hence in the transport cost estimates. Similar estimates in higher dimensions could be very useful to detect possible phase transitions, e.g.\ a phase transition in the parameter $\beta$ for the equilibrium measures of the infinite dimensional system of interacting SDEs studied by Osada \cite{Os12} (in dimension one these measures are conjectured to be - and proven to be for $\beta=1,2,4$ - the $\mathsf{Sine}_\beta$ processes).
Therefore we end the introduction with the following open problem:

\begin{openproblem}
Is it possible to establish similar results in higher dimension; e.g.\ reducing the finiteness of transportation cost or otherwise to the question of aysmptotics of moments?
\end{openproblem}

\section{Preliminaries}\label{s:prel}
We write $\lambda^1=\lambda$ and denote by $(\Omega,\F,\P)$ a generic probability space on which our random elements are defined. Given a map $S$ and a measure $\rho$ we denote the push-forward of $\rho$ by $S$ by $S_*\rho=\rho\circ S^{-1}$. The set of all $\sigma$- finite measure on a space $X$ will be denoted by $\mathcal M(X)$. For a Polish space $X$ we denote by $\mathcal B(X)$ its Borel $\sigma$- algebra. For $X=X_1\times X_2$ we denote the projection on $X_i$ by $\proj_i$.

\subsection{Random measures}

 Let $\mu$ be a random $\sigma$- finite measure on $\R,$ i.e.\ a measurable map $\mu:\Omega\to\mathcal M(\R).$ We assume that $\R$ acts on $(\Omega,\F)$ via a \emph{measurable flow} $\theta_t:\Omega\to\Omega,t\in\R$, i.e.\ the mapping $(\omega,t)\mapsto \theta_t\omega$ is $\F\otimes\mathcal B(\R)-\F$ measurable with $\theta_0=\id$ and $\theta_t\circ\theta_s=\theta_{t+s}$ for $s,t\in\R$. A random measure $\mu$ on $\R$ is then called \emph{invariant} (sometimes also \emph{equivariant}) if for $A\in\mathcal B(\R), t\in\R$ and $\omega\in\Omega$ it holds that
$$ \mu(\theta_t\omega,A-t)=\mu(\omega,A).$$
A random measure $q$ on $\R\times \R$ will be called invariant if for all $A,B\in\mathcal B(\R), t\in\R$ and $\omega\in\Omega$ it holds that
$$q(\theta_t\omega,A-t,B-t)=q(\omega,A,B).$$
For an invariant measure $\mu$ we sometimes write $\theta_t\mu(\omega)=\mu(\theta_t\omega).$

The \emph{intensity} of an invariant random measure $\mu$ on $\R$ is defined as $\E[\mu([0,1))]$; $\mu$ has unit intensity if $\E[\mu([0,1))]=1$.

A measure $\P$ on $(\Omega,\F)$ is called \emph{stationary} if it is invariant under the action of $\theta$, i.e.\ $\P\circ\theta_t=\P$ for all $t\in\R.$ 
\begin{remark}
 We can think of $\theta_t\omega$ as $\omega$ shifted by $-t$, see Example 2.1 in \cite{LaTh09}.
\end{remark}
From now on we will always assume to be in the setting described above.

So, let $\P$ be a stationary measure and $\mu$ be an invariant random measure. Let $B\in\mathcal B(\R)$ with $0<\lambda(B)<\infty$. The \emph{Palm measure} $\P_\mu$ of $\mu$ (with respect to $\P$) is the measure on $(\Omega,\F)$  defined by
\begin{align*}
\P_\mu(A):=\frac{1}{\lambda(B)} \E \int_B 1_A(\theta_t\omega)~\mu(\omega,dt). 
\end{align*}
As this is independent of $B$, we can deduce by a monotone class argument the \emph{refined Campbell theorem}
\begin{align}\label{eq:refined campbell}
 \E \int f(\theta_t \omega,t)\ \mu(\omega,dt)= \int_\Omega \int_\R f(\eta,s)\ ds\ \P_\mu(d\eta)
\end{align}
for bounded and measurable $f:\Omega\times\R\to\R.$ We refer to \cite[Chapter 8]{Th00} and \cite{La10} for more details on Palm theory.

Last and Thorisson \cite[Propostion 4.5]{LaTh09} show the following remarkable result which is crucial for the proof of Theorem \ref{thm:p=1}.
\begin{theorem}\label{thm:prop4.2}
 Consider two invariant random measures $\xi$ and $\eta$ and let $T:\Omega\times\R\to\R$ be measurable and satisfy 
$$T(\theta_t\omega,s-t)=T(\omega,s)-t\quad s,t\in\R,\ \omega\in\Omega.$$
Then $\P$-a.s. $T_*\xi=\eta$ iff for all $A\in\F$
$$\P_\xi(\theta_{T(0)}\omega\in A)=\P_\eta(A).$$
\end{theorem}
Any map $T$ as in the theorem will be called \emph{allocation rule} or invariant transport map.

\begin{example}\label{ex:key obs}
If $\P$ is stationary, the constant invariant random measure $\lambda$ has Palm measure $\P_\lambda=\P$. In particular, given an invariant random measure $\mu$ with unit intensity and an invariant transport map from $\lambda$ to $\mu$ which is measurably dependent only on the $\sigma$-algebra generated by $\mu$ Theorem \ref{thm:prop4.2} yields a \emph{shift-coupling}, see \cite{AlTh93} and \cite{Th94}, between $\P$ and $\P_\mu$, i.e.\ for all $A\in\F$ it holds that $\P[\theta_{T(0)}\omega\in A]=\P_\mu[A].$ By considering the image measure $\P\circ\mu^{-1}$ we can assume w.l.o.g.\ that $(\Omega,\F)$ is the canonical probability space $(\mathcal M(\R),\mathcal B(\mathcal M(\R))$ and $\mu$ the identity map. Then, Theorem \ref{thm:prop4.2} can be read as a shift-coupling between $\mu$ and $\P_\mu$:
$$\P[\theta_{T(0)}\mu\in \cdot]=\P_\mu[\cdot].$$
\end{example}

\subsection{Optimal transport between random measures}\label{sec:ot}

A \emph{semicoupling} between two measures $\nu$ and $\eta$ on $\R$ is a measure $q$ on $\R\times\R$ such that $(\proj_1)_*q\leq \nu$ and $(\proj_2)_*q=\eta.$ It is called \emph{coupling} if additionally $(\proj_1)_*q=\nu$. A semicoupling between $\lambda$ and a random measure $\mu$ is a random measure $q:\Omega\to\mathcal M(\R\times\R)$ such that for all $\omega\in\Omega$ the measure $q_\omega$ is a semicoupling between $\lambda$ and $\mu_\omega$. It is called coupling if additionally $q_\omega$ is a coupling between $\lambda$ and $\mu_\omega$ for all $\omega\in\Omega$. We denote the set of all couplings (resp.\ semicouplings) between $\lambda$ and $\mu$ by $\mathsf{Cpl}(\lambda,\mu)$ (resp.\ $\mathsf{SCpl}(\lambda,\mu)$).

Considering the \emph{cost-function} $c_p(x,y)=|x-y|^p$ for $0<p\leq 1$ we will be interested in the cost functional
$$\W_p(\nu,\eta):=\inf_{q\in\mathsf{SCpl}(\nu,\eta)}\E\int |x-y|^p\ q(dx,dy).$$
By standard results in optimal transport, e.g.\ \cite[Section 7.1]{Vi03}, $\W_p$ constitutes a metric as soon as $\P[\nu(\R)=\eta(\R)<\infty]=1$. 

Let $\mu$ be an invariant random measure with unit intensity. For $q\in\mathsf{SCpl}(\lambda,\mu)$ we set 
$$\mathfrak C(q)= \sup_{n\geq 1} \frac1n\int_{\R\times [0,n)}|x-y|^p q(dx,dy).$$
By \cite[Corollary 6.5]{Hu15}, we have
\begin{align*}
& \inf_{q\in\mathsf{SCpl}(\lambda,\mu)}\mathfrak{C}(q)\\
= & \liminf_{n\to\infty} \frac1n \inf_{q\in\mathsf{SCpl}(\lambda,\mu)}\int_{\R\times [0,n)}|x-y|^p q(dx,dy)\\
= &  \inf_{q\in\mathsf{SCpl}(\lambda,\mu)}\liminf_{n\to\infty} \frac1n \int_{\R\times [0,n)}|x-y|^p q(dx,dy)
=:  \mathfrak c_\infty.
\end{align*}
We sometimes write $\mathfrak c_\infty(p)$ to stress the dependence on $p$.

\begin{definition}
Let $\mu$ be an invariant random measure with unit intensity. A (semi)coupling $q$ between $\lambda$ and $\mu$ is called
\begin{itemize}
 \item \emph{asymptotically optimal} if $\mathfrak C(q)=\mathfrak c_\infty.$
\item \emph{optimal} if it is asymptotically optimal and invariant.
\end{itemize}
\end{definition}

The main results of \cite{HuSt13, Hu15} show that there is a unique optimal coupling between $\lambda$ and $\mu$ \emph{provided} that $\mathfrak c_\infty<\infty.$ In particular, eventhough there are arbitrarily many asymptotically optimal couplings there is a unique invariant one. Moreover, the optimal coupling $\hat q$ is concentrated on an invariant transport map $T$, i.e.\ $\hat q=(\id, T)_*\lambda,$ which is measurably only dependent on the $\sigma$-algebra generated by the random measure $\mu$.

\section{Proof of Theorem \ref{thm:p<1}}

\subsection{Proof of Theorem \ref{thm:p<1} i)}

The strategy is to construct a coupling between $\lambda$ and $\mu$ which is not optimal but whose cost can be controlled nicely.
To this end, we set $Z_n:=\mu([0,2^n))$ and put
$$ \bar\c_n:=2^{-n}\W_p(\1_{[0,Z_n)}\lambda,\1_{[0,2^n)}\mu).$$
By invariance of $\lambda$ and $\mu$ this equals $\frac12(\bar\c_n+\bar\c_n')$ with $Z_n':=\mu([2^n,2^{n+1}))=Z_{n+1}-Z_n$ and
$$ \bar\c_n':=2^{-n}\W_p(\1_{[2^n,2^n+Z_n')}\lambda,\1_{[2^n,2^{n+1})}\mu).$$
By the triangle inquality for $\W_p$ we have
\begin{align*}
& \bar\c_{n+1}-\bar\c_n = \bar\c_{n+1}- \frac12(\bar\c_n+\bar\c_n') \\
&= 2^{-(n+1)}\left(\W_p(\1_{[0,Z_{n+1})}\lambda,\1_{[0,2^{n+1})}\mu)-\W_p(\1_{[0,Z_n)}\lambda,\1_{[0,2^n)}\mu)- \W_p(\1_{[2^n,2^n+Z_n')}\lambda,\1_{[2^n,2^{n+1})}\mu)   \right) \\
&\leq 2^{-(n+1)} \W_p(\1_{[0,Z_{n+1})}\lambda,\1_{[0,Z_{n})}\lambda+\1_{[2^n,2^n+Z_n')}\lambda).
\end{align*}
The last expression can be estimated as follows. As $r\mapsto r^p$ is concave (recall $0<p<1$) the optimal coupling does not transport the common mass. Hence, in case that $Z_n\leq 2^n$ we have to transport mass of amount $2^n-Z_n$ at most distance $2^n-Z_n+Z_n'$. In case that $Z_n>2^n$ we have to transport mass of amount $Z_n-2^n$ at most distance $Z_n-2^n+Z_n'$. Therefore, we can estimate using H\"older's inequality
\begin{align*}
& \W_p(\1_{[0,Z_{n+1})}\lambda,\1_{[0,Z_{n})}\lambda+\1_{[2^n,2^n+Z_n')}\lambda)\\
\leq\ & \E\left[|Z_n-2^n|(|Z_n-2^n|+Z_n')^p\right]\\
\leq\ & \E\left[|Z_n-2^n|^{1+p}+|Z_n-2^n|(Z_n')^p\right]\\
\leq\ & \Var(Z_n)^{(1+p)/2}+ \Var(Z_n)^{1/2}\E[(Z_n')^{2p}]^{1/2}\\
\leq\ & \Var(Z_n)^{(1+p)/2}+ \Var(Z_n)^{1/2} (\Var(Z_n)+\E[Z_n]^2)^{p/2}\\
\leq\ & \Var(Z_n)^{(1+p)/2}+ \Var(Z_n)^{(1+p)/2}+ \Var(Z_n)^{1/2} 2^{np},
\end{align*}
where we used the identity $\Var(Z)=\E[Z^2]-\E[Z]^2$ in the second to last step and the inequality $(x+y)^p\leq x^p+y^p$ in the second as well as in the last step. Therefore, we get
$$\bar\c_{n+1}-\bar\c_n\leq 2^{-n}\Var(Z_n)^{(1+p)/2} + \frac12\Var(Z_n)^{1/2} 2^{n(p-1)},$$
which readily implies
\begin{lemma}\label{lem:cvg}
 If $\sum_{n\geq 1} 2^{-n}\Var(Z_n)^{(1+p)/2} + \frac12 \Var(Z_n)^{1/2} 2^{n(p-1)}<\infty$ then $\c_\infty(p)<\infty.$
\end{lemma}

\begin{proof}
 Put 
$$\c_n:=\inf_{q\in \mathsf{SCpl}(\lambda,\mu)} 2^{-n}\cdot \E\left[\int_{\R\times [0,2^n)} |x-y|^p \ q_\omega(dx,dy)\right]. $$
Then, we have $\c_n\leq \bar\c_n$ and hence $\c_\infty=\liminf_{n\to\infty}\c_n\leq\liminf_{n\to\infty}\bar\c_n=:\bar\c_\infty.$ Therefore, it is sufficient to show that $\bar\c_\infty<\infty.$ However, this follows from
$$\liminf\bar\c_n\leq \bar\c_N + \sum_{k\geq N} 2^{-k}\Var(Z_k)^{(1+p)/2} + \frac12\Var(Z_k)^{1/2} 2^{k(p-1)}$$
which is finite by assumption.
\end{proof}

\begin{proof}[Proof of Theorem \ref{thm:p<1} i)]
 Assume that $\limsup_{n\to\infty}\sqrt{\Var(\mu([0,n)))} \cdot n^{p-1}=0$. We have to verify the condition of Lemma \ref{lem:cvg}. By our assumption there is $N\in\N$ such that for all $n\geq N$  we have $\Var(Z_n)=\Var(\mu([0,2^n)))\leq 2^{2n(1-p)}.$ Hence, for $0<q<p$ we have
\begin{align*}
& \sum_{k\geq N} 2^{-k} \Var(Z_k)^{(1+q)/2} + \frac12 \Var(Z_k)^{1/2} 2^{k(q-1)}\\
\leq\ &\sum_{k\geq N} 2^{-k}2^{k(1-p)(1+q)} + \frac12 2^{k(q-1)-k(p-1)}\\
=\ & \sum_{k\geq N} 2^{k((1-p)(1+q)-1)}   +\frac12 2^{k(q-p)}<\infty,
\end{align*}
because $(1-p)(1+q)<(1-p)(1+p)=1-p^2<1$ and $q<p$. Hence, $\bar\c_\infty(q)<\infty.$
\end{proof}

\begin{remark}
 Note that we just showed that an equivalent condition in Lemma \ref{lem:cvg} would be that $\sum_{k\geq 1}\frac12 \Var(Z_k)^{1/2} 2^{k(p-1)}<\infty$. It is also not hard to see that the convergence of $\sum_{k\geq 1}\frac12 \Var(Z_k)^{1/2} 2^{k(p-1)}$ is strictly stronger than the convergence of $\sum_{k\geq N} 2^{-k}\Var(Z_k)^{(1+q)/2}$ in the sense that the convergence  of the second sum does not imply the convergence of the first.
\end{remark}

\subsection{Proof of Theorem \ref{thm:p<1} ii)}

\medskip

Denote by $q_n$ the optimal semicoupling between $\lambda$ and $\1_{[0,n)}\mu$. By Proposition 4.2 in \cite{Hu15}, there is a transport map $T_n$ and a density $\rho_n$ such that $q_n=(id,T_n)_*(\rho_n\lambda)$. Put $l_n:=\inf\{x:\rho_n(x)>0\}$ and $r_n:=\sup\{x:\rho_n(x)>0\}.$ If $l_n<0$ (resp. $n<r_n$) it follows by optimality that $\rho_n=1$ on $[l_n,0]$ (resp. $[n,r_n]$). In that case, we put 
$$a_n:= T_n(l_n/2), \quad(\text{resp.\ } b_n:=T_n(n+\frac12(r_n-n)).$$
If $l_n\geq 0$ (resp. $r_n\leq n$) we put $a_n=0$ (resp. $b_n=n$).
\medskip\\
We claim that there exists a sequence of events $(A_n)_n$ s.t.
\begin{itemize}
 \item[a)] $\liminf_{n\to\infty} \P[A_n]\geq c >0.$
\item[b)] on $A_n$ either $|l_n|\geq2\sqrt{\Var(\mu([0,n))}$ or $|r_n-n|\geq2\sqrt{\Var(\mu([0,n))}$
\item[c)] on $A_n$  there exists $1>\kappa>0$ such that for large $n$ either $a_n\geq \kappa n$ or $c_n:=n-b_n\geq \kappa n,$ i.e.\  $\liminf_{n\to\infty} (a_n+c_n)/n\geq \kappa.$
\end{itemize}
As a consequence of concavity of $r\mapsto r^p$ we have $T_n(x)\geq T_n(y)$ for all $l_n\leq x\leq y\leq 0$ (resp. $T_n(x)\leq T_n(y)$ for all $n\leq y\leq x\leq r_n$), e.g. see \cite{GaMc96}. Hence, assuming a),b) and c), we can argue
\begin{align*}
 \mathfrak c_\infty\ &\geq\ \liminf_{n\to\infty} \frac1n \E\left[\1_{A_n}\int_{[l_n,\frac{l_n}{2}]\cup[n+\frac{r_n-n}{2},r_n]} |x-T_n(x)|^p \lambda(dx)\right]\\
&\geq\ \liminf_{n\to\infty} \frac1n \kappa^p n^p  \sqrt{\Var(\mu([0,n))} \P[A_n]\\
&\geq\ \liminf_{n\to\infty} \kappa^p n^{p-1}  \sqrt{\Var(\mu([0,n))}\cdot c =\infty,
\end{align*}
by assumption. Hence, it remains to establish the claim.

We put $Y_n:=\mu([0,n))$ and set
$$\tilde A_n=\{ Y_n \geq n + 4\sqrt{\Var(Y_n)}\}.$$
By the CLT, it follows that $\liminf_{n\to\infty}\P[\tilde A_n]\geq \tilde c >0$ so that a) holds. On $\tilde A_n$ we have to transport mass of amount at least $4\sqrt{\Var(Y_n)}$ into the interval $[0,n]$. Hence, either $|l_n|\geq2\sqrt{\Var(Y_n)}$ or $|r_n-n|\geq2\sqrt{\Var(Y_n)}$ so that b) holds also. It remains to show c). We will show that on $\tilde A_n$ it is not possible that both $(a_k)_k\in  o(k)$ and $(c_k)_k\in o(k).$ Put $Y_{a_n}=\mu([0,a_n))$ and $Y_{c_n}'=\mu([b_n,n))$. Then, we have
\begin{align*}
 &\P[\tilde A_n,(a_k)_k\in o(k),(c_k)_k\in o(k)]\\
\leq\ & \P[Y_{a_n}+Y_{c_n}'\geq a_n + c_n + 2\sqrt{\Var(Y_n)}, (a_k)_k\in o(k),(c_k)_k\in o(k)],
\end{align*}
since on $\tilde A_n$ there is no transport from outside of $(T_n(l_n),T_n(r_n))$ into $(T_n(l_n),T_n(r_n))$, by concavity of the cost function, and at most half of the Lebesgue mass that is transported from outside of $[0,n]$ (the total excess is at least $4\sqrt{\Var(\mu([0,n))}$) is transported into $(\tilde a_n,T(l_n)]\cup [T(r_n),\tilde b_n)$ (where $\tilde a_n= a_n$ if $a_n>0$ and $\tilde a_n=T(l_n)$ otherwise and similarly for $\tilde b_n$). Hence,
\begin{align*}
 &\P[\tilde A_n,(a_k)_k\in o(k),(c_k)_k\in o(k)]\\
\leq\ & \P[Y_{a_n}\geq a_n +\sqrt{\Var(Y_n)}, (a_k)_k\in o(k)]+\P[Y_{c_n}'\geq c_n +\sqrt{\Var(Y_n)}, (c_k)_k\in o(k)]
\end{align*}
We consider these two terms seperately and start with the first one. We put $\P^{a}_n:=(a_n)_*(\1_{(a_k)_k\in o(k)}\P)$ and set $a_n^*:=\sup\{x:x\in \mbox{supp}(\P^a_n)\}\in o(n)$. Then, we have
\begin{align*}
  & \P[Y_{a_n}\geq a_n +\sqrt{\Var(Y_n)}, (a_k)_k\in o(k)]\\
\leq\ & \frac1{\Var(Y_n)}\E[(Y_{a_n}-a_n)^2,(a_k)_k\in o(k)]\\
\leq\ & \frac1{\Var(Y_n)} \int \Var(Y_t) \ \P^a_n(dt)\\
\leq\ & \frac{\Var(Y_{a_n*})}{\Var(Y_n)}=\frac{f(a_n^*)}{f(n)},
\end{align*}
which goes to zero by the assumption that $\mu$ has a regular variance.

The term $\P[Y_{c_n}'\geq c_n +\sqrt{\Var(Y_n)}, (c_k)_k\in o(k)]$ can be treated analogously. Hence, 
$$\P[\tilde A_n,(a_k)_k\in o(k),(c_k)_k\in o(k)]\to 0.$$
By making the sets $\tilde A_n$ slightly smaller yielding sets $A_n'$ we can therefore assume that for large $n$, say $n>N$, on $A_n'$ either $(a_k)_k\in\Theta(k)$ or $(c_k)_k\in \Theta(k)$ (since $a_n,c_n\leq n$), property b) holds and  $\liminf P[ A_n']\geq \tilde c/2=c'$. This means that for any $\omega\in A_n'$ there is $\kappa'(\omega)>0$ such that for large $n$ we have either $a_n(\omega)\geq \kappa'(\omega) n$ or $c_n(\omega)\geq \kappa'(\omega) n$. In particular, $\{\kappa'>0\}\supset \tilde A_n$ for all $n>N$. Take $\kappa>0$ such that $\P[\kappa'<\kappa]\leq \tilde c'/2$ and set $A_n:= A_n'\cap\{\kappa'\geq\kappa\}$. Then $(A_n)_{n\geq N}$ satisfy the required properties a),b) and c).

\section{Proof of Theorem \ref{thm:p=1}}
As indicated in the introduction the proof follows from the reasoning as in Section 3 of \cite{Li02} together with Proposition 4.5 of \cite{LaTh09}. Let $\P$ be some stationary measure on $\Omega$, $X$ be some real valued random variable and $\P':=\theta_X\P$, i.e.\ $\P$ and $\P'$ are \emph{shift-coupled} by $X$ (cf.\ \cite{AlTh93, Th00}). Then we have for any $f:\Omega\to [-1,1]$
\begin{align*}
& \left|\frac1t\int_0^t f(\omega)(\theta_{-s})_*\P(d\omega)-(\theta_{-s})_*\P'(d\omega))ds\right|\\
=\ &\left|\frac1t\int_0^t \E\left[f(\theta_{-s}\omega)-f(\theta_{-s+X}\omega)\right]ds \right|\\
=\ &\left|\frac1t \E\left[ \int_\R f(\theta_{-s}\omega) (\1_{[0,t]}(s) - \1_{[X,X+t]}(s) ) ds\right] \right|\\
\leq\ & \frac1t\E\left[\int_R\left|\1_{[0,t]}(s) - \1_{[X,X+t]}(s)\right|ds\right]\\
=\ & \frac2t\E[|X|\wedge t].
\end{align*}
Hence, we have derived the \emph{shift-coupling inequality}
$$\left\|\frac1t\int_0^t ((\theta_{-s})_*\P)-(\theta_{-s})_*\P') ds\right\|\leq \frac2t\E[|X|\wedge t],$$
where $\|\cdot\|$ denotes the total variation distance.

By Theorem \ref{thm:prop4.2}, any invariant transport map $T$ \emph{balancing} $\lambda$ and $\mu$, i.e.\ transporting $\lambda$ to $\mu$, induces a shift-coupling of $\P$ with its Palm-measure $\P_\mu$. By \eqref{eq:cost2}, $\c_\infty=\inf_{T, T_*\lambda=\mu}\E[|T|]$ and, by the results of \cite{Hu15}, the infimum is attained by a unique map $\hat T$ which is measurably dependent only on the $\sigma$- algebra generated by $\mu$. Hence, $X:=\hat T(0)$ shift-couples $\P$ and $\P_\mu$ and, by \eqref{eq:cost2}, we need to show that $\E[|X|]=\infty.$

By stationarity of $\P$ we have $\frac1t\int_0^t (\theta_{-s})_*\P ds=\P$ and by the refined Campbell theorem \eqref{eq:refined campbell} it follows that for any bounded and non-negative function $f:\Omega\to\R$ and $g(\omega,s):=\1_{[0,t]}(s)f(\theta_{-s}\omega)$ 
we have
\begin{align*}
\int_\Omega\int_0^t f(\omega)\ (\theta_{-s})_*\P_\mu(d\omega) ds = \int_\Omega \int_0^t f(\theta_{-s}\omega)\ \P_\mu(d\omega)ds\\
=\int_\Omega \int_\R g(\omega,s)~\P_\mu(d\omega)ds=\int f(\omega)\mu_\omega([0,t])~\P(d\omega)
 \end{align*}
Hence, we have $\int_0^t (\theta_{-s})_*\P_\mu(d\omega)ds=\mu_\omega([0,t])\P(d\omega)$. Putting everything together, we get (recall $\|fd\nu-gd\nu\|=\int|f-g|d\nu$)
\begin{align*}
 \left\|\frac1t\int_0^t ((\theta_{-s})_*\P - (\theta_{-s})_*\P_\mu) ds\right\|=\E\left[\left|1-\frac{\mu_\omega([0,t])}{t}\right|\right]\leq \frac2t\E_\mu[|X|\wedge t].
\end{align*}
By assumption, we have $\limsup_{t\to\infty}\E\left[\left|t-\mu_\omega([0,t])\right|\right]=\infty$. This implies
$$\E[|X|]\geq\limsup_{t\to\infty}\E[|X|\wedge t]\geq\limsup_{t\to\infty}\frac12\E\left[\left|t-\mu_\omega([0,t])\right|\right]=\infty, $$
which proves the result.

\begin{remark}
Following the argumentation in \cite[Section 3]{Li02} we can recover the assertion of Theorem \ref{thm:p<1} ii) in the setting of Theorem \ref{thm:p=1} assuming additionally that $\E_\mu[|t-\mu([0,t))|]\sim \sqrt{\Var(\mu([0,t)))}$ for large $t$. Indeed, we have with $Z_t=\mu([0,t))$ 
$$\frac2t\E_\mu[|X|\wedge t]\geq \frac1t\E_\mu[|t-Z_t|]\sim \frac1t \sqrt{\Var(Z_t)}.$$
By assumption, we have $\limsup_{t\to\infty} \sqrt{\Var(Z_t)}t^{p-1}\geq C>0.$ Therefore, we have for $t$ large enough
$$\sqrt{\Var(Z_t)}\geq C' t^{1-p}.$$
This implies
$$0<C'\leq \E\left[\frac{|X|\wedge t}{t^{1-p}}\right] \leq \E[|X|^p],$$
since $\frac{|X|\wedge t}{t^{1-p}}\leq |X|^p$. Assuming $\E[|X|^p]<\infty$ implies by the dominated convergence theorem that
 $$0<C'\leq\E\left[\frac{|X|\wedge t}{t^{1-p}}\right]\to 0\quad \text{ as }t\to\infty,$$
which is a contradiction.
\end{remark}

\end{document}